\documentclass{elsarticle}

\usepackage{amscd}
\usepackage{amsmath}
\usepackage{amssymb}
\usepackage{amsthm}
\usepackage{graphicx}
\usepackage[all]{xy}
\usepackage{kotex}
\usepackage{color}
\usepackage{enumerate}
\theoremstyle{plain}

\usepackage{nameref}
\newtheorem{Thm}{Theorem}[section]
\newtheorem{Lemma}[Thm]{Lemma}

\newtheorem*{Cor*}{Corollary}

\newtheorem*{definition*}{Definition}
\newtheorem*{Thm*}{Theorem}
\theoremstyle{remark}
\newtheorem{Rmk}[Thm]{\bf{Remark}}

\usepackage{color}							

\journal{Journal of Differential Equations}
\begin{document}
\begin{frontmatter}
\title{
A New Condition for the Concavity Method of Blow-up Solutions to p-Laplacian Parabolic Equations
}

\author[syc]{Soon-Yeong Chung\corref{cjhp}}
\address[syc]{Department of Mathematics and Program of Integrated Biotechnology, Sogang University, Seoul 04107, Korea}
\ead{sychung@sogang.ac.kr}

\author[mjc]{Min-Jun Choi}
\address[mjc]{Department of Mathematics, Sogang University, Seoul 04107, Korea}
\cortext[cjhp]{Corresponding author}
\ead{dudrka2000@sogang.ac.kr}

\begin{abstract}

In this paper, we consider an initial-boundary value problem of the p-Laplacian parabolic equations
\begin{equation*}
	\begin{cases}
		u_{t}\left(x,t\right)=\mbox{div}(|\nabla u\left(x,t\right)|^{p-2}\nabla u(x,t))+f(u(x,t)), & \left(x,t\right)\in \Omega\times\left(0,+\infty\right),\\
		u\left(x,t\right)=0, & \left(x,t\right)\in\partial \Omega\times\left[0,+\infty\right),\\
		u\left(x,0\right)=u_{0}\geq0, & x\in\overline{\Omega},
	\end{cases}
\end{equation*}
where $p\geq2$ and $\Omega$ is a bounded domain of $\mathbb{R}^{N}$ $(N\geq1)$ with smooth boundary $\partial\Omega$. The main contribution of this work is to introduce a new condition
\begin{center}
	$(C_{p})$$\hspace{1cm} \alpha \int_{0}^{u}f(s)ds \leq uf(u)+\beta u^{p}+\gamma,\,\,u>0$
\end{center}
for some $\alpha, \beta, \gamma>0$ with $0<\beta\leq\frac{\left(\alpha-p\right)\lambda_{1, p}}{p}$, where $\lambda_{1, p}$ is the first eigenvalue of p-Laplacian $\Delta_{p}$, and we use the concavity method to obtain the blow-up solutions to the above equations. In fact, it will be seen that the condition $(C_{p})$ improves the conditions ever known so far.
\end{abstract}

\end{frontmatter}
\section{Introduction}\label{introduction}

In this paper, we discuss the blow-up solutions for the following p-Laplacian parabolic equations
\begin{equation}\label{equation}
\begin{cases}
u_{t}\left(x,t\right)=\mbox{div}(|\nabla u\left(x,t\right)|^{p-2}\nabla u(x,t))+f\left(u\left(x,t\right)\right), & \left(x,t\right)\in \Omega\times\left(0,+\infty\right),\\
u\left(x,t\right)=0, & \left(x,t\right)\in\partial \Omega\times\left[0,+\infty\right),\\
u\left(x,0\right)=u_{0}\left(x\right)\geq0, & x\in\overline{\Omega},
\end{cases}
\end{equation}
where $p\geq2$, $\Omega$ is a bounded domain in $\mathbb{R}^{N}(N\geq1)$ with smooth boundary $\partial \Omega$ and $f$ is locally Lipschitz continuous on $\mathbb{R}$, $f(0)=0$ and $f(u)>0$ for $u>0$. Moreover, the initial data $u_{0}$ is assumed to be a non-negative and non-trivial function in $C^{1}(\overline{\Omega})$ with $u_{0}(x)=0$ on $\partial \Omega$ for $p=2$ and in $L^{\infty}(\Omega)\cap W_{0}^{1,p}(\Omega)$ for $p>2$, respectively.

There are many literatures dealing with a local existence of classical solutions (or weak solutions)  to the equations \eqref{equation}. In general, it is well known that not all solutions of the equations \eqref{equation} exist for all time. So, many authors have focused on the sufficient conditions for the local existence of solutions to the equations \eqref{equation}. In particular, for $p=2$, Ball \cite{B} derived sufficient conditions for the local existence solutions to the equations \eqref{equation}. On the other hand, for $p>2$, Zhao \cite{Z} also derived sufficient conditions for the nonexistence of global solutions to the equations \eqref{equation}.

On the other hand, the blow-up solutions to the equations \eqref{equation} have been studied by many authors. In particular, Levine and Payne \cite{L} studied the abstract equation
\[
\begin{cases}
P\frac{du}{dt}=-A(t)u + f(u(t)), & t\in [0, +\infty),\\
u(0)=u_{0},
\end{cases}
\]
where $P$ and $A$ are positive linear operators defined on a dense subdomain $D$ of a real or complex Hilbert Space, in which they obtained the blow-up solutions, under abstract conditions\\
\begin{equation}\label{abstract (C)}
2(\alpha + 1)F(x)\leq(x, f(x)),\,\,\,F(u_{0}(x))>\frac{1}{2}(u_{0}(x),A u_{0}(x))
\end{equation}
for every $x\in D$, where $F(x)=\int_{0}^{1}(f(\rho x), x)d\rho$. This work has been recognized as a creative and elegant tool for giving criteria for the blow-up, which is called ``the concavity method''. They also applied the method to some other equations or system of equations (See \cite{LP, LPa}).

Afterwards, the method in the abstract form was changed into a concrete form by Philippin and Proytcheva \cite{PP} and applied to the same equation as \eqref{equation} with $p=2$. In fact, the condition \eqref{abstract (C)} was changed into the form
\[
\begin{aligned}
\hbox{$(A)$ \hspace{1cm} $(2+\epsilon)F(u)\leq uf(u),\,\,u>0$},
\end{aligned}
\]
for some $\epsilon>0$ and the initial data $u_{0}$ satisfies
\begin{equation}\label{initial data 1}
-\frac{1}{2}\int_{\Omega}|\nabla u_{0}(x)|^{2}dx + \int_{\Omega}F(u_{0}(x))dx>0,
\end{equation}
where $F(u)=\int_{0}^{u}f(s)ds$.

Since then, the concavity method has been used so far to derive the blow-up solutions the variants of the equations \eqref{equation} or some other equations.

For example, Ding and Hu \cite{DH} adopted the condition (A) and
\[
-\int_{\Omega}\int_{0}^{|\nabla u_{0}|^{2}}\rho(y)dydx + 2k(0)\int_{\Omega}\int_{0}^{u_{0}}f(s)dsdx>0
\]

     to get blow-up solutions to the equation
\[
(g(u))_{t}=\nabla\cdot(\rho(|\nabla u|^{2})\nabla u) + k(t)f(u),
\]

assuming more that $k(0)>0$, $k'(t)\geq0$, $\lim_{s\rightarrow0^{+}} s^{2}g'(s)=0$, $g''(s)\leq0$, and
\[
0<s\rho(s)\leq(1+\alpha)\int_{0}^{s}\rho(y)dy.
\]

Another example is the work by Payne et al. in \cite{PPP, PS} in which they obtained the blow-up solutions to the equations
\begin{equation}\label{IBBequation}
\begin{cases}
u_{t}=\Delta u - g(u), & \hbox{in}\,\,\Omega\times(0,+\infty),\\
\frac{\partial u}{\partial n}=f(u), & \hbox{on}\,\, \partial \Omega\times(0,+\infty),\\
u(x,0)=u_{0}(x)\geq0,
\end{cases}
\end{equation}
when the Neumann boundary data $f$ satisfies the condition (A).

On the other hands, the condition $(A)$ for the nonlinear term $f$ and the condition \eqref{initial data 1} for the initial data $u_{0}$ were relaxed by Bandle and Brunner \cite{BB} as follows:

\[
\begin{aligned}
\hbox{(B) \hspace{1cm} $(2+\epsilon)F(u)\leq uf(u) + \gamma ,\,\,u>0$}
\end{aligned}
\]
for some $\epsilon>0$ and the initial data $u_{0}$ satisfying
\begin{equation}\label{initial data 2}
-\frac{1}{2}\int_{\Omega}|\nabla u_{0}(x)|^{2}dx+\int_{\Omega}[F(x,u_{0})-\gamma]dx>0,
\end{equation}


Concerning the general case $p>2$, in 1993, Zhao \cite{Z} studied the following equations
\begin{equation}\label{Zhao equation}
\begin{cases}
u_{t}=\mbox{div}(|\nabla u|^{p-2}\nabla u)+f(u), & \hbox{ in } \Omega\times\left(0,T\right),\\
u\left(x,t\right)=0, & \hbox{ on }\partial \Omega,\\
u\left(x,0\right)=u_{0}, & \hbox{ in } \Omega
\end{cases}
\end{equation}
and proved the blow-up solutions to the equations \eqref{Zhao equation} under the condition
\begin{equation}\label{pf<uf}
p\,F(u)\leq uf(u),\,\,u>0,
\end{equation}
for some $\epsilon>0$ and  the initial data $u_{0}$ satisfying
\begin{equation}\label{initial data 3}
-\frac{1}{p}\int_{\Omega}|\nabla u_{0}(x)|^{p}dx+\int_{\Omega}F(u_{0}(x))dx\geq  \frac{4(p-1)}{T(p-2)^{2}p}\int_{\Omega}u_{0}^{2}(x)dx
\end{equation}

In $2002$,  Messaoudi \cite{Me} proved that the solutions to the same equations \eqref{equation} blow up under the condition

\[
\begin{aligned}
\hbox{$(A_{p})$ \hspace{1cm} $(p+\epsilon)F(u)\leq uf(u),\,\,u>0$},
\end{aligned}
\]
and the initial data $u_{0}$ satisfying

\[
-\frac{1}{p}\int_{\Omega}|\nabla u_{0}(x)|^{p}dx + \int_{\Omega}F(u_{0}(x))dx>0.
\]


Looking into the above conditions $(A)$, $(A_{p})$, \eqref{pf<uf}, $(B)$ and so on more closely, we can see that they are independent of the eigenvalue of the Laplacian $\Delta$ (or p-Laplacian $\Delta_{p}$), which is a constant depending on the domain $\Omega$. From this point of view, there is, we think, a possibility that the above conditions can be improved and  refined in a way, depending on the domain and the eigenvalue.


Being motivated by this point, we introduce a new condition as follows: for some $\alpha, \beta, \gamma>0$,
\[
\hbox{$(C_{p})$\hspace{1cm}  $\alpha \int_{0}^{u}f\left(s\right)ds\leq uf(u) + \beta u^{p} + \gamma,\,\,u>0$},
\]
where $0<\beta\leq\frac{\left(\alpha-p\right)\lambda_{1, p}}{p}$ and $\lambda_{1, p}$ is the first eigenvalue of the Laplacian $\Delta_{p}$. (For simplicity, the condition $(C_{p})$ and $\lambda_{1,p}$ are denoted by $(C)$ when $p=2$, respectively.)\\

The main theorems of this paper are as follows:
\begin{Thm*}[Case 1 : $p=2$]\label{cBlowB}
	Let a function $f$ satisfy the condition $(C)$. If the initial data $u_{0}\in C^{1}(\overline{\Omega})$ with $u_{0}=0$ on $\partial \Omega$ satisfies
	\begin{equation}\label{11}
	-\frac{1}{2}\int_{\Omega}\left|\nabla u_{0}\left(x\right)\right|^{2}dx+\int_{\Omega}\left[\int_{0}^{u_{0}\left(x\right)}f(s)ds-\gamma\right]dx>0,
	\end{equation}
	then the nonnegative classical solutions $u$ to the equations \eqref{equation} blows up at finite time $T^{*}$, in the sense of
	\[
	\lim_{t\rightarrow T^{*}}\int_{0}^{t}\int_{\Omega}u^{2}\left(x,s\right)ds=+\infty,
	\]
	where $\gamma$ is the constant in the condition $(C)$.
\end{Thm*}

\begin{Thm*}[Case 2 : $p>2$]\label{cBlowBp}
	Let a function $f$ satisfy the condition $(C_{p})$ and $p>2$. If the initial data $u_{0}\in L^{\infty}(\Omega)\cap W_{0}^{1,p}(\Omega)$ satisfies
	\begin{equation}\label{J_{p}(0)>0}
	-\frac{1}{p}\int_{\Omega}\left|\nabla u_{0}\left(x\right)\right|^{p}dx+\int_{\Omega}\left[F(u_{0}(x))-\gamma\right]dx>0,
	\end{equation}
	then the nonnegative weak solutions $u$ to the equations \eqref{equation} blows up at finite time $T^{*}$, in the sense of
	\[
	\lim_{t\rightarrow T^{*}}\int_{0}^{t}\int_{\Omega}u^{2}\left(x,s\right)ds=+\infty,
	\]
	where $\gamma$ is the constant in the condition $(C_{p})$.
\end{Thm*}

We organized this paper as follows: In Section \ref{section p=2}, we discuss, when $p=2$,  the blow-up classical solutions using concavity method with the condition $(C)$ and in Section \ref{BCMp}, when $p>2$, we discuss the blow-up weak solutions using the same method with the condition $(C_{p})$, which is the general case. Finally, in Section \ref{section 3 conditions}, we discuss the condition $(C_{p})$, comparing with the conditions $(A_{p})$, $(B)$, and \eqref{pf<uf} so on, together with the condition $J_{p}(0)>0$ for the initial data.\\


\section{Case 1 : $p=2$ and the classical solutions}\label{section p=2}
%

The local existence of the classical solutions to the equations \eqref{equation}  with the case $p=2$ is well known (See Ball \cite{B}). So, accepting the local existence,  we focus ourselves on the discussion of the blow-up of the classical solutions to the equation \eqref{equation} with $p=2$.\\


The following lemma is going to be useful in the proof of Theorem \ref{BlowB}.
\begin{Lemma}[\cite{KP, P}]\label{eigenvalue}
	There exist $\lambda_{1}>0$ and $\phi_{1}\in H_{0}^{1}(\Omega)$ with $\phi_{1}>0$ in $\Omega$ such that
	\[
	\begin{cases}
	-\Delta\phi_{1}\left(x\right)=\lambda_{1}\phi_{1}\left(x\right), & x\in \Omega,\\
	\phi_{1}\left(x\right)=0, & x\in\partial \Omega.\\
	\end{cases}
	\]
	Moreover, $\lambda_{1}$ is given by
	\[
	\lambda_{1} = \inf_{w\in H_{0}^{1}(\Omega)}\frac{\int_{\Omega}|\nabla w|^{2}dx}{\int_{\Omega}|w|^{2}dx}>0.
	\]
	In the above, we recall that the number $\lambda_{1}$ is the first eigenvalue of $\Delta$ and $\phi_{1}$ is a corresponding eigenfunction.
\end{Lemma}


Let us recall the condition $(C)$ : for some $\alpha, \beta, \gamma>0$,
\[
\hbox{$(C)$\hspace{1cm}  $\alpha \int_{0}^{u}f\left(s\right)ds\leq uf(u) + \beta u^{2} + \gamma,\,\,u>0$},
\]
where $0<\beta\leq\frac{\left(\alpha-2\right)\lambda_{1}}{2}$ and $\lambda_{1}$ is the first eigenvalue of the Laplacian $\Delta$ on $\Omega$. \\

\begin{Rmk}
	We will discuss the condition $(C)$ in the section \ref{section 3 conditions}, comparing with the condition $(A)$ and $(B)$ introduced in the first section, together with the condition $J(0)>0$ for the initial data.
\end{Rmk}

Now, we state and prove our result for $p=2$.

\begin{Thm}\label{BlowB}
Let a function $f$ satisfy the condition $(C)$. If the initial data $u_{0}\in C^{1}(\overline{\Omega})$ with $u_{0}=0$ on $\partial \Omega$ satisfies
\begin{equation}\label{11}
-\frac{1}{2}\int_{\Omega}\left|\nabla u_{0}\left(x\right)\right|^{2}dx+\int_{\Omega}\left[\int_{0}^{u_{0}\left(x\right)}f(s)ds-\gamma\right]dx>0,
\end{equation}
then the nonnegative classical solutions $u$ to the equations \eqref{equation} blows up at finite time $T^{*}$, in the sense of
\[
\lim_{t\rightarrow T^{*}}\int_{0}^{t}\int_{\Omega}u^{2}\left(x,s\right)ds=+\infty,
\]
where $\gamma$ is the constant in the condition $(C)$.
\end{Thm}
\begin{proof}
We first define a functional $J$ by
\[
J\left(t\right):=-\frac{1}{2}\int_{\Omega}\left|\nabla u\left(x,t\right)\right|^{2}dx+\int_{\Omega}\left[F\left(u\left(x,t\right)\right)-\gamma\right]dx,\;\;\;t\geq0,
\]
where $F(u):=\int_{0}^{u}f(s)ds$.

Then by \eqref{11},
\[
J\left(0\right)=-\frac{1}{2}\int_{\Omega}\left|\nabla u_{0}\left(x\right)\right|^{2}dx+\int_{\Omega}\left[F(u_{0}\left(x\right))-\gamma\right]dx>0.
\]
and we can see that
\begin{equation}\label{Jtt}
J(t)=J(0)+\int_{0}^{t}\frac{d}{dt}J(s)ds=J(0)+\int_{0}^{t}\int_{\Omega}u_{t}^{2}(x,s)dxds.
\end{equation}

Now, we introduce a new function
\begin{equation}\label{It}
I\left(t\right)=\int_{0}^{t}\int_{\Omega}u^{2}\left(x,s\right)dxds+M,\,t\geq 0,
\end{equation}
where $M>0$ is a constant to be determined later. Then it is easy to see that
\begin{equation}\label{15}
\begin{aligned}
I'\left(t\right) & =  \int_{\Omega}u^{2}\left(x,t\right)dx\\
& =  \int_{\Omega}\int_{0}^{t}2u\left(x,s\right)u_{t}\left(x,s\right)dsdx+\int_{\Omega}u_{0}^{2}\left(x\right)dx.
\end{aligned}
\end{equation}

Then we use integration by parts, the condition $(C)$, Lemma \ref{eigenvalue}, and \eqref{Jtt} in turn to obtain
\begin{equation}\label{14}
\begin{aligned}
I''\left(t\right)&= \frac{d}{dt}\int_{\Omega}u^{2}\left(x,t\right)dx\\
&=\int_{\Omega}2u\left(x,t\right)u_{t}\left(x,t\right)dx\\
&=\int_{\Omega}2u\left(x,t\right)\left[\Delta u\left(x,t\right) + f\left(u(x,t)\right) \right]dx\\
&=-2\int_{\Omega}\left|\nabla u\left(x,t\right)\right|^{2}dx + \int_{\Omega}2u(x,t)f\left(u(x,t)\right)dx\\
&\geq-2\int_{\Omega}\left|\nabla u\left(x,t\right)\right|^{2}dx + \int_{\Omega}2\left[\alpha F(u(x,t))-\beta u^{2}(x,t)-\alpha \gamma\right] dx\\
&=2\alpha\left[-\frac{1}{2}\int_{\Omega}\left|\nabla u(x,t)\right|^{2}dx+\int_{\Omega}[F(u(x,t))-\gamma] dx\right]\\
&+(\alpha-2)\int_{\Omega}\left|\nabla u\left(x,t\right)\right|^{2}dx - 2\beta \int_{\Omega}u^{2}(x,t)dx\\
&\geq 2\alpha J(t)+\left[(\alpha-2)\lambda_{1}-2\beta\right]\int_{\Omega}u^{2}(x,t)dx\\
&\geq 2\alpha \left[J(0)+\int_{0}^{t}\int_{\Omega}u_{t}^{2}(x,s)dxds\right].
\end{aligned}
\end{equation}

Using the Schwarz inequality, we obtain
\begin{equation}\label{16}
\begin{aligned}
& I'\left(t\right) ^{2}\\
& \leq 4\left(1+\delta\right)\left[\int_{\Omega}\int_{0}^{t}u\left(x,s\right)u_{t}\left(x,s\right)dsdx\right]^{2}+\left(1+\frac{1}{\delta}\right)\left[\int_{\Omega}u_{0}^{2}\left(x\right)dx\right]^{2}\\
& \leq  4\left(1+\delta\right)\left[\int_{\Omega}\left(\int_{0}^{t}u^{2}\left(x,s\right)ds\right)^{\frac{1}{2}}\left(\int_{0}^{t}u_{t}^{2}\left(x,s\right)ds\right)^{\frac{1}{2}}dx\right]^{2}\\
&+\left(1+\frac{1}{\delta}\right)\left[\int_{\Omega}u_{0}^{2}\left(x\right)dx\right]^{2}\\
& \leq  4\left(1+\delta\right)\left(\int_{\Omega}\int_{0}^{t}u^{2}\left(x,s\right)dsdx\right)\left(\int_{\Omega}\int_{0}^{t}u_{t}^{2}\left(x,s\right)dsdx\right)\\
&+\left(1+\frac{1}{\delta}\right)\left[\int_{\Omega}u_{0}^{2}\left(x\right)dx\right]^{2},
\end{aligned}
\end{equation}
where $\delta>0$ is arbitrary. Combining the above estimates \eqref{It}, \eqref{14}, and \eqref{16}, we obtain that for $\sigma=\delta=\sqrt{\frac{\alpha}{2}}-1>0$,
\begin{equation*}
\begin{aligned}
&I''\left(t\right)I\left(t\right)-\left(1+\sigma\right)I'\left(t\right)^{2}\\
& \geq 2\alpha\left[J\left(0\right)+\int_{0}^{t}\int_{\Omega}u_{t}^{2}\left(x,s\right)dxds\right]\left[\int_{0}^{t}\int_{\Omega}u^{2}\left(x,s\right)dxds+M\right]\\
& -  4\left(1+\sigma\right)\left(1+\delta\right)\left[\int_{\Omega}\int_{0}^{t}u^{2}\left(x,s\right)dsdx\right]\left[\int_{\Omega}\int_{0}^{t}u_{t}^{2}\left(x,s\right)dsdx\right]\\
& -  \left(1+\sigma\right)\left(1+\frac{1}{\delta}\right)\left[\int_{\Omega}u_{0}^{2}\left(x\right)dx\right]^{2}\\
& \geq  2\alpha M\cdot J\left(0\right)-\left(1+\sigma\right)\left(1+\frac{1}{\delta}\right)\left[\int_{\Omega}u_{0}^{2}\left(x\right)dx\right]^{2}.
\end{aligned}
\end{equation*}

Since $J\left(0\right)>0$ by the assumption, we can choose $M>0$ to be large enough so that
\begin{equation}\label{17}
I''(t)I\left(t\right)-\left(1+\sigma\right)I'\left(t\right)^{2}>0.
\end{equation}

This inequality \eqref{17} implies that for $t\geq0$,
\[
\frac{d}{dt}\left[\frac{I'\left(t\right)}{I^{\sigma+1}\left(t\right)}\right]>0  \hbox{  i.e. }I'\left(t\right)\geq\left[\frac{I'\left(0\right)}{I^{\sigma+1}\left(0\right)}\right]I^{\sigma+1}\left(t\right).
\]

Therefore, it follows that $I\left(t\right)$ cannot remain finite for all $t>0$. In other words, the solutions $u$  blow up in finite time $T^{*}$.
\end{proof}

\begin{Rmk}\label{blow-up time p=2}
We estimate the blow-up time of the solutions to equation \eqref{equation} roughly. Put
\[
M:=\frac{\frac{\alpha}{\alpha-2}\left(1+\sqrt{\frac{\alpha}{2}}\right)\left[{\int_{\Omega}u_{0}^{2}\left(x\right)dx}\right]^{2}}{2\alpha\left[-\frac{1}{2}\int_{\Omega}\left|\nabla u_{0}\left(x\right)\right|^{2}dx+\int_{\Omega}\left[F(u_{0}\left(x\right))-\gamma\right]dx\right]}.
\]
Then we obtain that
\[
\begin{cases}
I'\left(t\right)\geq\left[\frac{\int_{\Omega}u_{0}^{2}\left(x\right)dx}{M^{\sigma+1}}\right]I^{\sigma+1}\left(t\right), & t>0,\\
I\left(0\right)=M,
\end{cases}
\]
which implies
\[
I\left(t\right)\geq\left[\frac{1}{M^{\sigma}}-\frac{\sigma\int_{\Omega}u_{0}^{2}\left(x\right)dx}{M^{\sigma+1}}t\right]^{-\frac{1}{\sigma}}
\]
where $\sigma=\sqrt{\frac{\alpha}{2}}-1>0$. Then the blow-up time $T^{*}$ satisfies
\begin{equation*}\label{22}
0<T^{*}\leq\frac{M}{\sigma\int_{\Omega}u_{0}^{2}\left(x\right)dx}.
\end{equation*}
\end{Rmk}


\section{Case 2 : $p>2$ and the weak solutions}\label{BCMp}

In this section, we discuss the blow-up of solutions to the equations \eqref{equation} for the case $p>2$, which is the main part of our work. In order to make this section self-contained we state, without proof, a local existence result of \cite{Z}.

%
%
%
%

\begin{Thm}[See \cite{Z}]
	Let $f$ be in $C(\mathbb{R})$ and there exists a function $g(u)\in C^{1}(\mathbb{R})$ such that
	\[
	\left|f(u)\right|\leq g(u).
	\]
	Then for any $u_{0}\in L^{\infty}(\Omega)\cap W_{0}^{1,p}(\Omega)$, there exists $T>0$ such that \eqref{equation} has a solution
	\[
    u\in L^{\infty}(\Omega\times(0,T))\cap L^{p}(0,T;W_{0}^{1,p}(\Omega)), u_{t}\in L^{2}(\Omega\times(0,T)).
	\]
\end{Thm}

The following lemmas are used when proving Theorem \ref{BlowBp}.

\begin{Lemma}[\cite{KP, P}]\label{eigenvalue_p}
	For $1<p<\infty$, there exist $\lambda_{1, p}>0$ and $\phi_{1, p}\in W_{0}^{1,p}(\Omega)$  with $\phi_{1, p}>0$ in $\Omega$  such that
	\[
	\begin{cases}
	-\Delta_{p} \phi_{1, p}(x)=\lambda_{1, p} |\phi_{1, p}(x)|^{p-2}\phi_{1, p}(x),  & x\in\Omega \\
	\phi_{1, p}\left(x\right)=0, & x\in\partial \Omega.
	\end{cases}
	\]
	Moreover, $\lambda_{1, p}$ is given by
	\[
	\lambda_{1, p} = \inf_{w\in W_{0}^{1,p}(\Omega)}\frac{\int_{\Omega}|\nabla w|^{p}dx}{\int_{\Omega}|w|^{p}dx}>0
	\]
	In the above, we recall that the number $\lambda_{1, p}$ is the first eigenvalue of $\Delta_{p}$ and $\phi_{1, p}$ is a corresponding eigenfunction.
\end{Lemma}

%
%

Let us recall that for some $\alpha, \beta, \gamma>0$,
\[
\hbox{$(C_{p})$\hspace{1cm}  $\alpha \int_{0}^{u}f\left(s\right)ds\leq uf(u) + \beta u^{p} + \gamma,\,\,u>0$},
\]
where $0<\beta\leq\frac{\left(\alpha-p\right)\lambda_{1, p}}{p}$ and $\lambda_{1, p}$ is the first eigenvalue of the $p$-Laplacian $\Delta_{p}$ on $\Omega$.

\begin{Rmk}
	We will discuss the condition $(C_{p})$ in the next section, comparing with the conditions $(A_{p})$ and $(B_{p})$ introduced in the first section, together with the initial data condition.
\end{Rmk}

Now, we state and prove our main result.

\begin{Thm}\label{BlowBp}
	Let a function $f$ satisfy the condition $(C_{p})$ and $p>2$. If the initial data $u_{0}\in L^{\infty}(\Omega)\cap W_{0}^{1,p}(\Omega)$ satisfies
	\begin{equation}\label{J_{p}(0)>0}
	-\frac{1}{p}\int_{\Omega}\left|\nabla u_{0}\left(x\right)\right|^{p}dx+\int_{\Omega}\left[F(u_{0}(x))-\gamma\right]dx>0,
	\end{equation}
	then the nonnegative weak solutions $u$ to the equations \eqref{equation} blows up at finite time $T^{*}$, in the sense of
	\[
	\lim_{t\rightarrow T^{*}}\int_{0}^{t}\int_{\Omega}u^{2}\left(x,s\right)ds=+\infty,
	\]
	where $\gamma$ is the constant in the condition $(C_{p})$.
\end{Thm}

The following lemma is essential in the proof of the above theorem.


\begin{Lemma}[\cite{Z}]\label{inequalities}
    Let $u$ be the weak solutions to the equations \eqref{equation} with $|\nabla u_{0}|\in L^{p}(\Omega)$. Then
	\begin{enumerate}
	\item[\emph{(i)}]
    \begin{align*}
    \int_{0}^{t}\int_{\Omega}\frac{1}{2}(u^{2}(x,s))_{t}dxds&=\frac{1}{2}\int_{\Omega}\left[u^{2}(x,t)-u_{0}^{2}(x)\right]dx\\
    &=\int_{0}^{t}\int_{\Omega}\left[-\left|\nabla u\left(x,t\right)\right|^{p}+u(x,t)f(u(x,t))\right]dxds
    \end{align*}
    \item[\emph{(ii)}]
    \begin{align*}
    \int_{0}^{t}\int_{\Omega}u_{t}^{2}(x,s)dxds&\leq -\frac{1}{p}\int_{\Omega}\left[\left|\nabla u\left(x,t\right)\right|^{p}-\left|\nabla u_{0}(x)\right|^{p}\right]dx\\
    &+\int_{\Omega}\left[F(u(x,t))-F(u_{0}(x))\right]dx,
    \end{align*}
    \end{enumerate}
    where $F(u):=\int_{0}^{u}f(s)ds$.
\end{Lemma}

\emph{Proof of Theorem \ref{BlowBp}.}
	We define a function $J_{p}$ by
	\begin{equation}
	J_{p}(t):=-\frac{1}{p}\int_{\Omega}\left|\nabla u(x,t)\right|^{p}dx+\int_{\Omega}[F(u(x,t))-\gamma] dx.
	\end{equation}
	Then it follows from \eqref{J_{p}(0)>0} and Lemma \ref{inequalities} (ii) that,
	\[
	J_{p}(0)=-\frac{1}{p}\int_{\Omega}\left|\nabla u_{0}\left(x\right)\right|^{p}dx+\int_{\Omega}\left[F(u_{0}(x))-\gamma\right]dx>0
	\]
	and 
	\begin{equation}\label{J(t) estimate}
	\begin{aligned}
	J_{p}(t)&=-\frac{1}{p}\int_{\Omega}\left|\nabla u(x,t)\right|^{p}dx+\int_{\Omega}[F(u(x,t))-\gamma] dx\\
	&\geq -\frac{1}{p}\int_{\Omega}\left|\nabla u_{0}(x)\right|^{p}dx + \int_{\Omega}\left[F(u_{0}(x))-\gamma\right]dx+ \int_{0}^{t}\int_{\Omega}u_{t}^{2}(x,s)dxds,\\
	&=J_{p}(0)+\int_{0}^{t}\int_{\Omega}u_{t}^{2}(x,s)dxds,
	\end{aligned}
	\end{equation}
	
	On the other hand, we define a function $I_{p}$ by
	\begin{equation}\label{Ip}
	I_{p}\left(t\right):=\int_{0}^{t}\int_{\Omega}u^{2}\left(x,s\right)dxds+M,\,t\geq 0,
	\end{equation}
	where $M>0$ is a constant to be determined later. Then it is easy to see that
	\begin{equation}
	\begin{aligned}
	I_{p}'\left(t\right) & =  \int_{\Omega}u^{2}\left(x,t\right)dx\\
	& = \int_{0}^{t}\int_{\Omega}2u\left(x,s\right)u_{t}\left(x,s\right)dsdx+\int_{\Omega}u_{0}^{2}\left(x\right)dx.
	\end{aligned}
	\end{equation}

	Then by Lemma \ref{inequalities} (i), we can see that
	\begin{equation}
	\begin{aligned}
	I_{p}''\left(t\right)&= \int_{\Omega}(u^{2}(x,t))_{t}dx\\
	&=2\int_{\Omega}\left[-\left|\nabla u\left(x,t\right)\right|^{p}+u(x,t)f(u(x,t))\right]dx.
	\end{aligned}
	\end{equation}

	By using the condition ($C_{p}$), Lemma \ref{eigenvalue_p}, and \eqref{J(t) estimate} in turn, we obtain that
	\begin{equation}\label{I''p}
	\begin{aligned}
	I_{p}''\left(t\right)&\geq-2\int_{\Omega}\left|\nabla u\left(x,t\right)\right|^{p}dx + 2\int_{\Omega}\left[\alpha F(u(x,t))-\beta u^{p}(x,t)-\alpha \gamma\right] dx\\
	&=2\alpha J_{p}(t) +\frac{2(\alpha-p)}{p}\int_{\Omega}\left|\nabla u\left(x,t\right)\right|^{p}dx - 2\beta \int_{\Omega}u^{p}(x,t)dx\\
	&\geq 2\alpha J_{p}(t) + 2\left[\frac{(\alpha-p)\lambda_{1, p}}{p} - \beta \right]\int_{\Omega}u\left(x,t\right)^{p}dx\\
	&\geq 2\alpha J_{p}(t)\\
	&\geq 2\alpha\left[J_{p}(0)+\int_{0}^{t}\int_{\Omega}u_{t}^{2}(x,s)dxds\right].
	\end{aligned}
	\end{equation}
	Applying the Schwarz inequality, as done in Theorem \ref{BlowB}, we obtain that
	\begin{equation}\label{I'p}
	\begin{aligned}
	I_{p}'\left(t\right) ^{2} & \leq  4\left(1+\delta\right)\left(\int_{\Omega}\int_{0}^{t}u^{2}\left(x,s\right)dsdx\right)\left(\int_{\Omega}\int_{0}^{t}u_{t}^{2}\left(x,s\right)dsdx\right)\\
	&+\left(1+\frac{1}{\delta}\right)\left[\int_{\Omega}u_{0}^{2}\left(x\right)dx\right]^{2},
	\end{aligned}
	\end{equation}
	where $\delta>0$ is arbitrary. Combining the above estimates \eqref{Ip}, \eqref{I''p}, and \eqref{I'p}, we obtain that
	\begin{equation}\label{I''I-(p)I''2}
	I_{p}''(t)I_{p}(t)-(1+\sigma)I_{p}'(t)^{2}>0,
	\end{equation}
	by choosing $\sigma=\delta=\sqrt{\frac{\alpha}{2}}-1>0$ and $M$ to be large enough. This means that the solutions $u$ blow up in finite time $T^{*}$.
	\begin{flushright}
	$\Box$
	\end{flushright}

\begin{Rmk}
	\begin{enumerate}
		\item[(i)] In a same way as in Remark \ref{blow-up time p=2}, the blow-up time $T^{*}$ can be estimated as follows
		\begin{equation*}\label{22}
		0<T^{*}\leq\frac{\frac{\alpha}{\alpha-2}\left(1+\sqrt{\frac{\alpha}{2}}\right)\left[{\int_{\Omega}u_{0}^{2}\left(x\right)dx}\right]^{2}}{2\alpha\sigma\left[-\frac{1}{p}\int_{\Omega}\left|\nabla u_{0}\left(x\right)\right|^{p}dx+\int_{\Omega}\left[F(u_{0}\left(x\right))-\gamma\right]dx\right]\int_{\Omega}u_{0}^{2}\left(x\right)dx}.
		\end{equation*}
		\item[(ii)] In fact, the above theorem and its proof can still works for the case $p=2$. 
	\end{enumerate}
	
\end{Rmk}

\section{Discussion on the Condition $(C_{p})$ with the initial data conditions}\label{section 3 conditions}

In this section, we compare the conditions $(A_{p})$ and $(C_{p})$ each other and discuss the role of the condition $J_{p}(0)>0$ for the initial data $u_{0}$.


As seen in the proof of Theorem \ref{BlowBp}, the concavity method is a tool for deriving the blow-up solution via the auxiliary function $J_{p}(t)$ under the condition $(A_{p})$ or $(C_{p})$, by imposing $J_{p}(0)>0$, instead of the large initial data.

On the other hand, instead of the condition $(B)$ in Section \ref{introduction}, it is not difficult to consider $(B_{p})$, in a similar form as in $(A_{p})$ or $(C_{p})$. In fact, to be strange, the condition $(B_{p})$ is not seen in any literature, as far as the authors know.

Then for $p\geq2$, let us recall the conditions as follows:

for some $\epsilon, \beta, \hbox{ and } \gamma>0$,
\[
\begin{aligned}
&\hbox{$(A_{p})$ $\hspace{1cm} (p+\epsilon) F(u)\leq uf(u)$},\\
&\hbox{$(B_{p})$ $\hspace{1cm} (p+\epsilon) F(u)\leq uf(u) + \gamma$},\\
&\hbox{$(C_{p})$ $\hspace{1cm} (p+\epsilon) F\left(u\right) \leq uf(u)+\beta u^{p}+\gamma$},
\end{aligned}
\]
for every $u>0$, where $0<\beta\leq\frac{\epsilon\lambda_{1, p}}{p}$ and $F(u):=\int_{0}^{u}f(s)ds$. Here, note that the constants $\epsilon, \beta, \hbox{ and } \gamma>0$ may be different in each case.\\

Then it is easy to see that $(A_{p})$ implies $(B_{p})$ and $(B_{p})$ implies $(C_{p})$, in turn. The difference between $(B_{p})$ and $(C_{p})$ is whether or not they depend on the domain. The condition $(B_{p})$ is independent of the first eigenvalue $\lambda_{1, p}$ which depends on the domain $\Omega$. However, the condition $(C_{p})$ depends on domain, due to the term $\beta u^{p}$ with $0<\beta\leq\frac{\epsilon\lambda_{1, p}}{p}$. From this point of view, the condition $(C_{p})$ can be understood as a refinement of $(B_{p})$, corresponding to the domain. On the contrary, if a function $f$ satisfies $(C_{p})$ for every bounded domain $\Omega$ with smooth boundary $\partial \Omega$, then the first eigenvalue $\lambda_{1, p}$ can be arbitrary small so that the condition $(C_{p})$ get closer to $(B_{p})$ arbitrarily. Besides, as far as the authors know, there has not been any noteworthy condition for the concavity method other than $(A_{p})$ or $(B_{p})$. \\

On the other hand, using the fact that $(C_{p})$ is equivalent to

\begin{equation*}\label{F(u)/u^2+m}
\frac{d}{du}\left(\frac{F(u)}{u^{p+\epsilon}}-\frac{\gamma}{p+\epsilon}\cdot\frac{1}{u^{p+\epsilon}}-\frac{\beta}{\epsilon}\cdot\frac{1}{u^{\epsilon}}\right)\geq0,\,\,u>0,
\end{equation*}

we can easily see that for every $u>0$,

\begin{equation}\label{conditionsss}
\begin{aligned}
&\hbox{$(A_{p})$ holds if and only if $F(u)={u^{p+\epsilon}}h_{1}(u)$},\\
&\hbox{$(B_{p})$ holds if and only if $F(u)={u^{p+\epsilon}}h_{2}(u)+b$},\\
&\hbox{$(C_{p})$ holds if and only if $F(u)={u^{p+\epsilon}}h_{3}(u)+au^{p}+b$},
\end{aligned}
\end{equation}
for some constants $\epsilon>0$, $a>0$, and $b>0$ with $0<a\leq\frac{\lambda_{1, p}}{p}$, where $h_{1}$, $h_{2}$, and $h_{3}$ are nondecreasing function on $(0,+\infty)$. Here also, the constants $\epsilon, a, \hbox{ and } b$ may be different in each case. We note here that the nondecreasing functions $h_{1}$ is nonnegative on $(0,+\infty)$, but  $h_{2}$ and $h_{3}$ may not be nonnegative, in general.


\begin{Lemma}\label{C condition f>=a u^1+e}
	Let $f$ be a function satisfying $(C_{p})$ and $f(u)\geq \lambda u^{p-1}$, $u>0$, where $\lambda > \lambda_{1, p}$. Then the condition $(C_{p})$ implies that there exists $m>0$ such that $h_{3}(u)>0$ for $u>m$. In this case, we can find $\mu>0$  such that $f(u)\geq \mu u^{p-1+\epsilon}$, $u\geq m$. Moreover, the conditions $(B_{p})$ and $(C_{p})$ are equivalent.
\end{Lemma}

\begin{proof} First, it follows from \eqref{conditionsss} and the fact $\lambda > \lambda_{1, p}$ that $F(u)\geq \frac{\lambda}{p}u^{p}\geq \frac{\lambda_{1, p}}{p}u^{p}$ and so that
	\[
	u^{p+\epsilon}h_{3}(u)=F(u)-au^{p}-b\geq \frac{\lambda-\lambda_{1, p}}{p}u^{p}-b,
	\]	
	which goes to $+\infty$, as $u\rightarrow+\infty$. So, we can find $m>1$ such that $h_{3}(m)>0$, which implies that
	\[
	F(u)\geq u^{p+\epsilon}h_{3}(u),\,\, u\geq m.
	\]
	Putting it into the condition $(C_{p})$, we obtain
	\[
	u^{p+\epsilon}h_{3}(m)\leq uf(u) + \beta u^{p} + \gamma
	\]
	or
	\[
	u^{p-1+\epsilon}h_{3}(m)\leq f(u)  + \beta u+\frac{\gamma}{u}\leq (1 + \frac{\epsilon}{p})f(u) + \gamma,\,\,u\geq m>1,
	\]
	which gives
	\[
	f(u)\geq \mu u^{p-1+\epsilon},\,\,u\geq m>1
	\]
	for some $\mu>0$ and another constant $m$. \\
	
	Now, assume that the condition $(C_{p})$ is true. Since $0<\beta\leq\frac{\epsilon\lambda_{1, p}}{p}$ and $f(u)\geq \lambda u^{p-1}>\lambda_{1, p}u^{p-1}$, $u>0$, it follows from $(C_{p})$ that
	\[
	\epsilon_{1} F\left(u\right) + \left(p+\epsilon_{2}\right) F\left(u\right) \leq uf(u)+\frac{\epsilon\lambda_{1, p}}{p} u^{p}+\gamma,
	\]
	where $\epsilon_{1}=\frac{\epsilon\lambda_{1, p}}{\lambda}>0$ and $\epsilon_{2}=\epsilon-\epsilon_{1}>0$. This implies that for every $u>0$,
	\[
	\begin{aligned}
	uf(u)+\gamma &\geq \left(p+\epsilon_{2}\right) F\left(u\right) + \epsilon_{1} \int_{0}^{u}\left[f\left(s\right)-\lambda s^{p-1}\right]ds\\
	&\geq \left(p+\epsilon_{2}\right) F\left(u\right),
	\end{aligned}
	\]
	which gives $(B_{p})$.

\end{proof}

\begin{Rmk}\label{Zhao junning}
In general, the constant $\alpha$ with $\alpha>p$ in $(C_{p})$ can not be replaced by $p$. But, assume $p>2$ and $f$ satisfies a condition,
\[
\hbox{$(C_{p})'$ $\hspace{1cm} p F(u)\leq uf(u) + \gamma$, $u>0$,}
\]
which comes from $(C_{p})$ by replacing $\alpha$ by  $p$  and  taking $\beta=0$. Then the inequalities \eqref{I''p} and \eqref{I''I-(p)I''2} in the proof of Theorem \ref{BlowBp} can be derived in an easy way as follows:
	\begin{equation*}
	\begin{aligned}
	I_{p}''\left(t\right)&\geq-2\int_{\Omega}\left|\nabla u\left(x,t\right)\right|^{p}dx + 2\int_{\Omega}\left[p F(u(x,t))-p \gamma\right] dx\\
	&=2p J_{p}(t)\\ 
	\end{aligned}
	\end{equation*}
and 	
	\begin{equation*}
	I_{p}''(t)I_{p}(t)-(1+\sigma)I_{p}'(t)^{2}>0.
	\end{equation*}

Therefore, we can prove that the weak solutions to the equations \eqref{equation} for $p>2$ blows up in a finite time, under the conditions $(C_{p})'$ and $J_{p}(0)>0$, which can be understood as an improvement of the result by Zhao \cite{Z}.
\end{Rmk}

\begin{Rmk}
	It is well known that if $\int_{m}^{+\infty}\frac{ds}{f(s)}=+\infty$ for some $m>0$, the solutions to equations \eqref{equation} is global. On the contrary, it has not been clear yet whether or not the condition $\int_{m}^{+\infty}\frac{ds}{f(s)}<+\infty$ guarantees the blow-up solution. Instead, when $f(u)\geq \mu u^{(p-1)+\epsilon}$, $u\geq m$ for some $\epsilon>0$ and $m>0$, the solutions to the equations \eqref{equation} blow up in a finite time, only if the initial data $u_{0}$ is sufficiently large (for more details, see \cite{LX}).\\
\end{Rmk}

In general, the condition $(C_{p})$ may not guarantee the blow-up solutions for any initial data $u_{0}$. In fact, we can easily see that a function $f(u)=au^{p-1}$ $(p>2)$ satisfies $(C_{p})$ if and only if $a\leq \lambda_{1, p}$. However, for any function $u_{0}$,
\[
\begin{aligned}
J(0)&=-\frac{1}{p}\int_{\Omega}\left|\nabla u_{0}\left(x\right)\right|^{p}dx+\int_{\Omega}\left[\frac{a}{p}u_{0}^{p}(x)-\gamma\right]dx\\
& \leq  \frac{a-\lambda_{1, p}}{p}\int_{\Omega}|u_{0}\left(x\right)|^{p}dx-\gamma|\Omega|\,<\,0,
\end{aligned}
\]
which means that there is no initial data $u_{0}$ satisfying $J(0)>0$, when $f(u)=au^{p-1}$, $a\leq\lambda_{1, p}$. Of course, it is well known that the solutions to the equations \eqref{equation} is global, in this case (see \cite{LX}).

So,  we are here going to discuss when we can find initial data $u_{0}$ satisfies $J(0)>0$. Consider a domain $\Omega$ with $\lambda_{1, p}>\frac{p}{p-1}$ and a nonnegative continuous function $f$ satisfying the condition $(C_{p})$ with $\gamma=1$ for simplicity and $f(s)\geq p\lambda_{1, p} s^{p-1}$, $s>0$. Now, let us take $u_{0}(x):=\phi_{1, p}(x)$ where $\phi_{1, p}$ is an eigenfunction in Lemma \ref{eigenvalue_p} with $\int_{\Omega}[\phi_{1, p}(x)]^{p}dx=|\Omega|$. Then it follows that

\[
\begin{aligned}
J(0)&=-\frac{1}{p}\int_{\Omega}\left|\nabla \phi_{1, p}\left(x\right)\right|^{p}dx+\int_{\Omega}\int_{0}^{\phi_{1, p}(x)
}f(s)ds-|\Omega|\\
&\geq-\frac{\lambda_{1, p}}{p}\int_{\Omega}[\phi_{1, p}\left(x\right)]^{p}dx+\int_{\Omega}\int_{0}^{\phi_{1, p}(x)
}p\lambda_{1, p}s^{p-1}\, ds-|\Omega|\\
&=\lambda_{1, p}\left(1-\frac{1}{p}\right)\int_{\Omega}[\phi_{1, p}(x)]^{p}dx-|\Omega|\\&=\left[\lambda_{1, p}\left(1-\frac{1}{p}\right)-1\right]|\Omega|>0.
\end{aligned}
\]

\section*{Conflict of Interests}
\noindent The authors declare that there is no conflict of interests regarding the publication of this paper.

\section*{Acknowledgments}
\noindent This work was supported by Basic Science Research Program through the National Research Foundation of Korea(NRF) funded by the Ministry of Education (NRF-2015R1D1A1A01059561).

\end{document}